\documentclass[12pt]{amsart}

\usepackage{hyperref}
\hypersetup{
	bookmarks=true,         % show bookmarks bar?
	unicode=false,          % non-Latin characters in Acrobat’s bookmarks
	pdftoolbar=true,        % show Acrobat’s toolbar?
	pdfmenubar=true,        % show Acrobat’s menu?
	pdffitwindow=false,     % window fit to page when opened
	pdftitle={},    % title
	pdfauthor={},     % author
	pdfsubject={},   % subject of the document
    pdfcreator={Me},   % creator of the document
	pdfproducer={},  % producer of the document
	pdfkeywords={}, % list of keywords
	pdfnewwindow=true,      % links in new window
	colorlinks=false,       % false: boxed links; true: colored links
	linkcolor=red,          % color of internal links
	citecolor=green,        % color of links to bibliography
	filecolor=magenta,      % color of file links
	urlcolor=cyan           % color of external links
}

\usepackage[dvipsnames]{xcolor}
\usepackage{amssymb,verbatim}
\usepackage{amsmath,amsfonts,enumitem}
\usepackage[mathscr]{euscript} % make a nice \Net
\usepackage{amsthm}
\usepackage{url}
\usepackage{graphicx} %to include figure
\usepackage{enumitem} %to use nice enumeration with changed indent
\usepackage{hyperref} % to make links for lemmas and theorems
\hypersetup{colorlinks} %so the links look nicer
\usepackage{bbm} % for the characteristic function 1
\usepackage{bm} % for bold greek letter
\usepackage{mathrsfs} %to use \mathscr
\usepackage{cancel} %to make sout to cross things out

\usepackage[colorinlistoftodos]{todonotes}

\RequirePackage{cleveref}
\usepackage{hypcap}
\hypersetup{colorlinks=true, citecolor=darkblue, linkcolor=darkblue}
\definecolor{darkblue}{rgb}{0.0,0,0.7}
\newcommand{\darkblue}{\color{darkblue}}

\definecolor{darkred}{rgb}{0.68,0,0}
\newcommand{\darkred}{\color{darkred}}

\definecolor{darkgreen}{rgb}{0,.38,0}
\newcommand{\darkgreen}{\color{darkgreen}}

\newcommand{\defn}[1]{\emph{\darkblue #1}}
\newcommand{\defna}[1]{\emph{\darkred #1}}
\newcommand{\defnb}[1]{\emph{\darkblue #1}}
\newcommand{\defng}[1]{\emph{\darkgreen #1}}

\setlist[enumerate]{
	label=\textnormal{({\roman*})},
	ref={\roman*}}

\makeatletter
\def\th@plain{%
	\thm@notefont{}% same as heading font
	\itshape % body font
}
\def\th@definition{%
	\thm@notefont{}% same as heading font
	\normalfont % body font
}
\makeatother

\newtheorem{thm}{Theorem}[section]
\newtheorem{lemma}[thm]{Lemma}

\newtheorem*{claim*}{Claim}
\newtheorem{cor}[thm]{Corollary}

\newtheorem{prop}[thm]{Proposition}

\newtheorem{conv}[thm]{Convention}
\newtheorem{question}[thm]{Question}

\newtheorem*{algSc}{Scaling Algorithm}
\newtheorem*{algFS}{FinStar Algorithm}

\theoremstyle{definition}
\newtheorem{ex}[thm]{Example}
\newtheorem{rem}[thm]{Remark}

\numberwithin{figure}{section}
\numberwithin{equation}{section}

\def\rr{\mathbb R}

\def\sm{\smallsetminus}

\def\De{\Delta}

\def\cT{\mathcal T}

\def\ssu{\subset}

\def\<{\langle}
\def\>{\rangle}

\def\0{{\mathbf 0}}

\def\st{{\rm st}}
\def\lk{{\rm lk}}

\def\ptl{{\partial}}

\def\.{\hskip.06cm}
\def\ts{\hskip.03cm}

\def\.{\hskip.06cm}
\def\ts{\hskip.03cm}

\def\nin{\noindent}

\def\hr{{\text {\rm hr}}}

\title{All triangulations have a common stellar subdivision}

\author[Karim Adiprasito]{Karim A.~Adiprasito}
\author[Igor Pak]{Igor Pak}
\address{Sorbonne Université and Université Paris Cité, CNRS, IMJ-PRG, F-75005 Paris, France}
\email{karim.adiprasito@imj-prg.fr}
\address{Department of Mathematics, %
	UCLA, Los Angeles, CA 90095, %
	USA}
\email{pak@math.ucla.edu}

\thanks{\today}

\begin{document}
	
	\maketitle

	\begin{center} \emph{For Frank, in memory}
\end{center}	
	
\begin{abstract}
We address two longstanding open problems, one originating in PL topology, another in birational geometry.
First, we prove the weighted version of Oda's \emph{strong factorization conjecture} (1978), 
and prove that every two birational toric varieties are related by a common iterated blowup 
(at rationally smooth points).
Second, we prove that every two PL homeomorphic polyhedra have a common stellar subdivisions, 
as conjectured by Alexander in~1930.
\end{abstract}

	\section{Introduction} \label{s:intro}

	Let $Q$ be a geometric complex in $\rr^d$, and let $T$ be a
	triangulation of~$Q$.  Define a \defn{stellar subdivision}
	at point $z \in Q$ to be a transformation given by adding to $T$
	cones over all faces in $T$ containing~$z$ (see Figure~\ref{f:stell-plane}).
    We say that a triangulation $T$ can be
	\defn{obtained by stellar subdivisions} \ts from a triangulation~$S$, if there
	is a finite sequence of stellar subdivisions which start at~$S$ and end with~$T$.
	When $S$ can be obtained by stellar subdivisions from triangulation~$T$
	and~$T'$, it is called a \defn{common stellar subdivision}
    (see Figure~\ref{f:stellex} below).
	
\begin{thm}[{\rm \defna{\em weighted strong factorization theorem}}{}] \label{t:Oda-space}
Every two triangulations $A, B$ of a geometric complex in~$\rr^d$,
have a common stellar subdivision.
Moreover, if both $A$ and $B$ have coordinates in a field extension
$K$ over $\mathbb{Q}$, then so does the common stellar subdivision.
\end{thm}
	
Over $\mathbb{Q}$, this implies the weighted version of \defn{Oda's conjecture} \ts
\cite{Oda78}, cf.~$\S$\ref{ss:finrem-unweighted}.
	
\begin{cor}\label{c:birational}
Every two birationally isomorphic toric varieties have a common toric
blowup (with blowups at rationally smooth points).
\end{cor}

The \defn{weak factorization conjecture} \ts states that every two triangulations
$A, B$ of a geometric complex in~$\rr^d$ are connected by a sequence
of stellar subdivisions and their inverses. This was proved in
dimension at most three in \cite{Dan83}, and in full generality
in \cite{Wlo97}, see also \cite{AKMW02,IS10}.
	
Morelli claimed the proof of the strong factorization conjecture in \cite{Mor96},
which was shown incorrect in \cite{Mat00}.  In a positive direction, the conjecture was
confirmed in \cite{Mac21} for a very special class of polyhedra.
In~\cite{DK11}, the authors proposed an algorithmic construction,
which remains unproven (cf.~$\S$\ref{ss:finrem-DK}).
Our approach is notably different, but is also constructive.
As an application, we obtain the following result.
	
\begin{thm}[{\rm \defna{\em  former Alexander's conjecture}}{}] \label{t:Alex}
Every two PL homeomorphic simplicial complexes
have combinatorially isomorphic stellar subdivisions.
\end{thm}
	
Alexander \cite{Ale30} was interested in PL homeomorphisms of polyhedral spaces,
and the theorem says that every two PL homeomorphic polyhedra have a common
stellar subdivision.  In this case, we do not have a geometric meaning,
but a topological one.
In dimension $d=2$, the conjecture was proved by Ewald \cite{Ewald}.
For the context of Alexander's conjecture, see e.g.\ \cite[$\S$4]{Lik99}.

It was noted by Anderson and Mn\"ev \cite{AM03}, that Theorem~\ref{t:Alex}
follows from Theorem~\ref{t:Oda-space}. We include a short proof in
Section~\ref{s:Alex} for completeness.  Finally, we
note that a topological version of the weak factorization conjecture 
was proved Alexander \cite{Ale30} (see also \cite{LN16,Pac91} 
and~$\S$\ref{ss:finrem-comp-top}).

	\medskip
	
	\section{Basic definitions and notation}\label{s:def}

Let $Q$ be a polyhedral complex embedded in~$\rr^d$.  We say that~$Q$
is a \defn{triangulation} if it is simplicial.	We use the same terms and
notation in both geometric (realized in the Euclidean space) and topological
setting (abstract complexes within the PL category), hoping this would not
lead to a confusion.  We use the terms ``geometric triangulation'',
``geometric (polyhedral) complex'', etc., when the distinction needs
to be emphasized. However, until Section~\ref{s:Alex}, we exclusively
work in the geometric setting.
	
Denote by $\cT(Q)$ the set of triangulations of~$Q$.
We write $S< T$ if $T$ is a \defn{refinement} of~$S$, where $S, T \in \cT(Q)$,
that is, if every simplex of $T$ is contained in a simplex of~$S$.
We write $S \lhd T$ if $T$ can be obtained from~$S$ by a
sequence of stellar subdivisions. In this case we say that
$T$ is an \defn{iterated stellar subdivision} of~$S$.
We will speak of \defn{common $($iterated$)$ stellar subdivision}
of triangulations $S,T\in \cT(Q)$ to mean a triangulation $R\in \cT(Q)$,
such that $S  \lhd R$ and  $T \lhd R$, see Figure~\ref{f:stellex}.

\begin{figure}[hbt]
		\includegraphics[width=15.1cm]{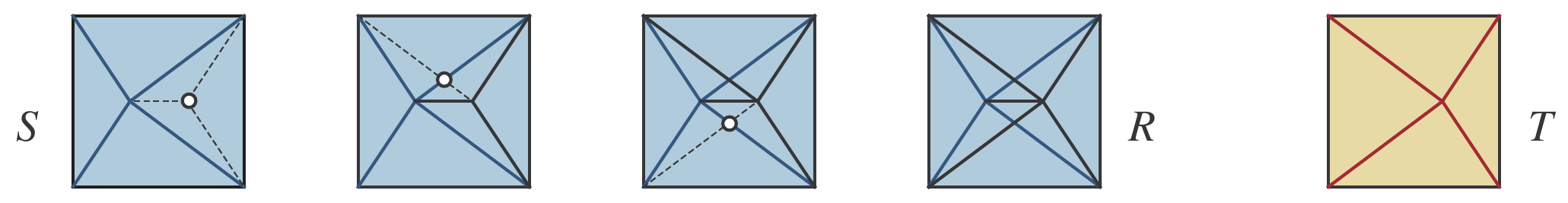}
		%\vskip-.25cm
		\caption{Triangulations $S,T$ of a square, a common stellar subdivision~$R$,
                    and stellar subdivisions from $S$ to~$R$.}
		\label{f:stellex}
\end{figure}

% For triangulations $R,S\in \cT(Q)$, we write $R \leftrightsquigarrow S$,
% if $S$ can be obtained by a sequence of stellar and inverse stellar
% subdivisions from~$R$. This is equivalent to stating that $R$ and $S$
% are PL homeomorphic (in the topological category) resp.\
% that the identity map is the desired PL homeomorphism
% (in the geometric setting).
	
Let $T$ be a simplicial subcomplex and let $F$ be a face of~$T$.
The \defn{star} \. $\st_F T$ is the minimal simplicial subcomplex
of $T$ that contains all faces containing~$F$.  The \defn{link} \. $\lk_F T:= \ptl \. \st_F T$
is the boundary of $\st_F T$ with respect to the intrinsic topology of~$T$.
We use $T-F$ to denote maximal subcomplex of~$T$ which does not contain~$F$,
also called the \defn{antistar} of~$F$ in~$T$.

	\medskip

\section{Planar case}\label{s:plane}

In this and the following two sections we are concerned \ts
{\em only} \ts with geometric triangulations.  In this section,
we consider triangulations of a convex polygon.  In the next two
sections, we consider geometric complexes in higher dimensions.

Note that in the plane, there are only two types of stellar subdivisions
shown in Figure~\ref{f:stell-plane} below.\footnote{Strictly speaking,
there is a third combinatorial type, when the added vertex is on the boundary.
To illustrate that,
simply delete the bottom triangle from the second type of stellar subdivision.}
The circle and dashed lines indicate the added vertices
and edges.  We use this notation throughout the paper
(see e.g.\ Figure~\ref{f:stellex} above).

\begin{figure}[hbt]
		\includegraphics[width=15.5cm]{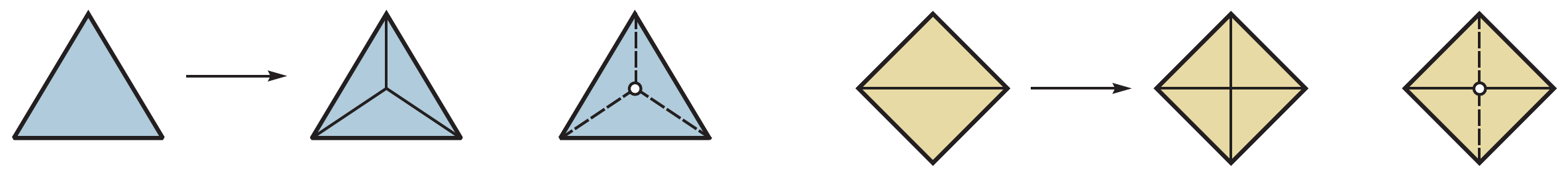}
		\vskip-.35cm
\caption{Two types of stellar subdivisions in the plane.}
\label{f:stell-plane}
\end{figure}

\subsection{Triangulations of polygons} \label{ss:plane-first}
%
%\subsection{First algorithm: triangulations of polygons} \label{ss:plane-first}
%
The case of $d=2$ is especially elegant since in this
case a triangulation of a convex polygon in the plane is a face
to face subdivision into triangles.  In this section we present
a self-contained proof of the weighted Oda conjecture in the plane.

Let $Q\ssu \rr^2$ be a convex polygon in the plane, and let $T \in \cT(Q)$
be a triangulation of~$Q$.  Let $x\in Q$ be a point in the relative
interior of a triangle $(abc)$ in~$T$, and let $T'$ be a triangulation
obtained from $T$ by adding edges $xa$, $xb$ and~$xc$.  Similarly, let
$x\in Q$ be a point in the relative interior of an edge~$ab$, and let
$T'$ be a triangulation obtained from $T$ by adding edges $xc$ for all
triangles $(abc)$ in~$T$.  A \defn{stellar subdivision}
is an operation $T\mapsto T'$ in both cases.  Clearly, we then have $T < T'$.

\begin{thm}[{\rm strong factorization for convex polygons}{}]
\label{t:Oda-plane-quant}
Suppose triangulations $T, T'$ of a convex polygon~$Q$
have at most $n$ vertices.  Then there is a triangulation $S\in \cT(Q)$
which can be obtained by a sequence of at most \. $30\ts n^3$ stellar
subdivisions from both $T$ and~$T'$.
\end{thm}

The theorem follows from the stellar subdivision algorithm we present
below.  
% We then give a modified version of the algorithm making it
% more cumbersome but amenable to generalization.

\medskip

\subsection{Stellar subdivision of fins} \label{ss:plane-find}
Let $Q\ssu \rr^2$ be a polygon in the plane, that is, a disk with polygonal
boundary, and let $V$ be its set of vertices.  Fix a vertex $v\in V$
which we call an \defn{anchor}.

We say that $Q$ is \defn{star-shaped} at anchor~$v$, if $[u,v]\ssu Q$ for all
$u \in Q$. We call $Q$ a \defn{strictly star-shaped} if for every point
$x$ in~$Q$, the line segment from $x$ to $v$ intersects the boundary
of $Q$ only in $v$ and possibly~$x$.
Denote by $\ptl Q$ the boundary of~$Q$.  For a region $D \ssu Q$,
denote by $T|_D$ the restriction of triangulation~$T$ to~$D$.
%
% Let $u,w\in V$ be the adjacent vertices of~$Q$.
% Denote by \. $C:=\partial Q - uv - vw$ \. the \emph{opposite boundary}.
%

Let $T \in \cT(Q)$ be a triangulation of a strictly star-shaped polygon~$Q$.
We call $T$ a \defn{fin} with respect to the anchor $v$.
We say a polytope $P$ is \defn{compatible} with a polyhedral complex $X$
if restricting $X$ to the faces $X|_{P}$ contained in $P$ is a subdivision of~$P$.

We think of $T$ at the set of triangles, and use $V_T$ and $E_T$ to denote
vertices and edges in~$T$, respectively.  We say that $T \in \cT(Q)$ is a
\defn{scaled fin} (triangulation) anchored at~$v$, if for every vertex $z\in V_T$,
the triangulation~$T$ is compatible with the line segment from $v$ to~$z$.

A scaled fin without interior vertices is called a \defn{stripe}. We also
consider the stripe associated to a fin $T$: It is the minimal stripe
containing all vertices of $T$.  An interesting case is the one when $T$
is a scaled fin, and $S$ is its stripe.

\begin{lemma}\label{l:fin1}
Let $v\in V$ be a vertex of the polygon $Q\ssu \rr^2$ which is star-shaped at~$v$.
Let $S,T \in \cT(Q)$ be scaled fins of~$Q$ anchored at~$v$, such that $S< T$
and that $S$ is the stripe of~$T$. Then \ts $S \lhd T$.
\end{lemma}

\begin{proof}
Use induction on the number $|V_T|$ of vertices in~$T$.
If $T$ has no interior vertices, we have $T=S$ and the result is trivial.
In general, suppose $uw\in E_T$ is an edge of~$T$ such that $u\in \ptl Q$
and $uw$ separates two triangles along the boundary~$\ptl Q$.
By going along the boundary, is easy to see that there exists
at least one such edge~$uw$.
	
There are two cases.  First, suppose $w\in \ptl Q$ and $uw$ separates $Q$ into
a triangle $\De$ and a polygon $Q'=Q\sm \De$.  Make a stellar move in~$S$
by adding the edge $uw$ to obtain a triangulation $S'$ of~$Q'$.  Since
$Q'$ is star-shaped at~$v$, this reduced the problem to triangulations
$S'$ and $T'= T\sm \De$, where $S' < T'$.
	
Second, suppose $w\in \ptl Q$ and $uw$ separates triangles $\De_1=(auw)$ and $\De_2=(buw)$
in~$Q$. Now collapse $\De_1\cup \De_2$\ts, i.e.\ let $Q' = Q\sm (\De_1\cup \De_2)$.
Make a stellar move in~$S$ by adding a vertex $u$ with edges $au$ and~$bu$,
to obtain a triangulation $S'$ of~$Q'$.  Since $Q'$ is star-shaped at~$v$,
this reduced the problem to triangulations $S'$ and $T'= T\sm (\De_1\cup \De_2)$,
where $S' < T'$.
	
\begin{figure}[hbt]
		\includegraphics[width=15.1cm]{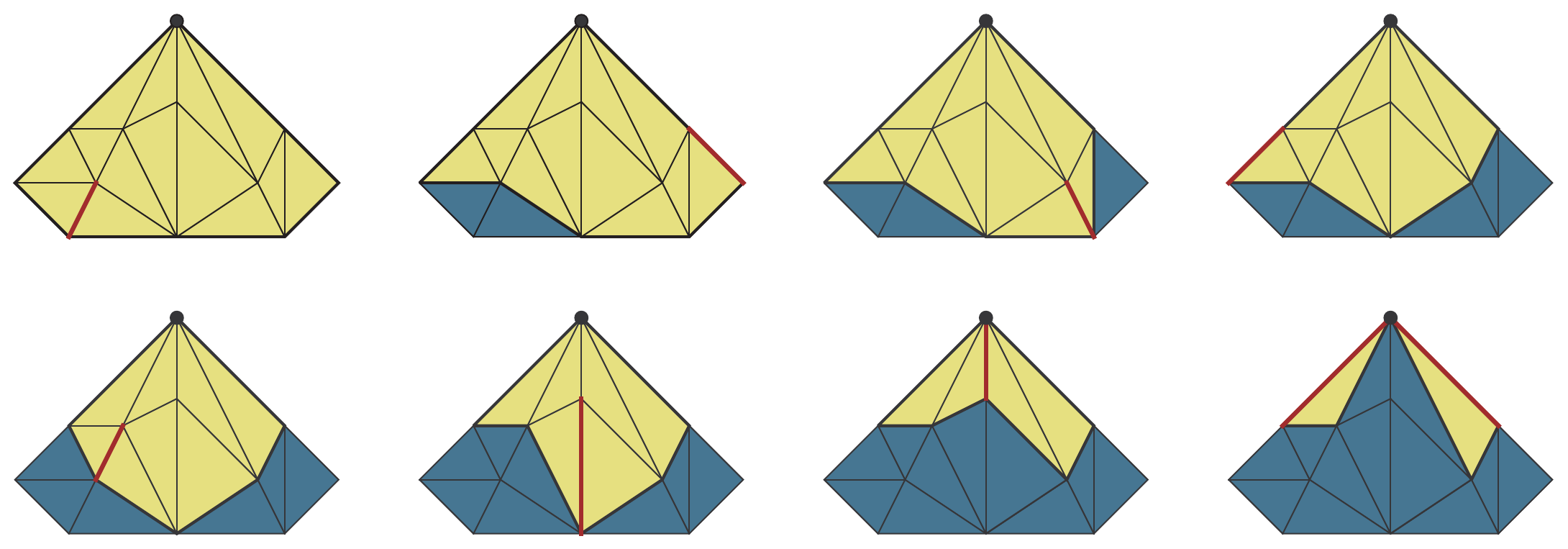}
		%\vskip-.25cm
		\caption{Shedding sequence for a scaled fin of a star-shaped polygon.
			Collapsed edges are shown in red.}
		\label{f:shed1}
\end{figure}
	
Note that in the second case polygon~$Q'$ can become connected in~$v$ only,
see Figure~\ref{f:shed1}.  This does not affect the argument, as one can
treat each component separately and proceed by induction. This completes
the proof.  \end{proof}

\begin{rem}\label{r:fin-number}
	The algorithm in the proof will be called the \defn{shedding routine}.
	Recall that for a triangulation $T$ with $|V_T|=n$ vertices, the number
	of edges $|E_T|\le 3n-6$ and the number of triangles $|T| \le 2n-4$.
	Thus, the number of stellar subdivisions used by the shedding
	routine is at most~$2n$.
\end{rem}

\medskip

\subsection{Common stellar triangulations in the plane}\label{ss:plane-plane}
To construct a common stellar triangulation, follows a series of steps.
Start with triangulations $A, B\in \cT(Q)$  of a convex polygon $Q$ in the plane.
Fix a vertex $v \in V$.  Let $|V_A|=m$ and $|V_B|=n$, so $m,n\ge 3$.

\medskip

\nin
\defng{\bf{Step~1.}}  Use stellar subdivisions in~$A$ to construct a scaled fin
triangulation $A'\rhd A$ that is anchored at~$v$.  Proceed as follows.
For every vertex $u\in (V_A-v)$, let $vw$ be an interval such that
$u\in vw$ and $w \in \ptl Q$. We will add such intervals one by one
in any order, until the desired scaled fin~$A$ is obtained.

\begin{figure}[hbt]
	\includegraphics[width=15.99cm]{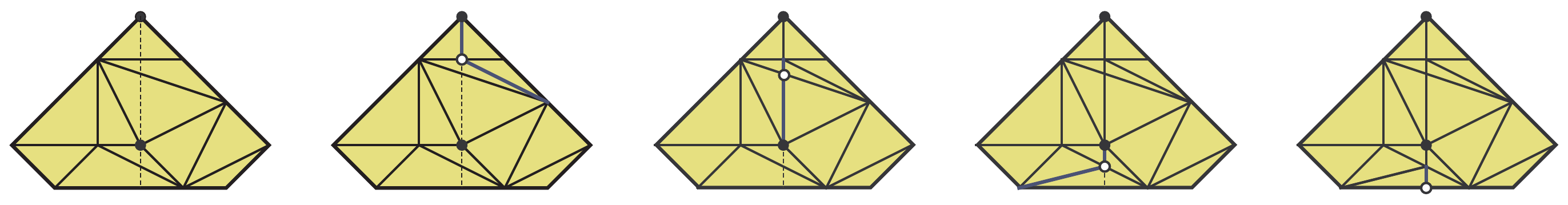}
	%\vskip-.25cm
	\caption{Adding a dotted line in Step~1 using stellar subdivisions.}
	\label{f:fin}
\end{figure}

To add $vw$, note that $vw$ intersects the existing edges $ab\in E_A$.
Make a stellar subdivision at points of intersection $vw \cap ab$.
Do this in the order from~$v$ towards~$w$.  At each
subdivision, the first added edge is along $vw$ while another
may be diverge.  The last of the intervals to be added is along $vw$
adjacent to~$w$, see Figure~\ref{f:fin}.

\medskip

\nin
\defng{\bf{Step~2.}} Use stellar subdivisions in~$B$ to construct a fin
triangulation $B'\rhd B$ that is anchored at~$v$ and refines~$A'$, i.e.\ $A'<B'$.
Proceed as follows.  First, add all vertices $u \in V_{A'}$ one by one, by making
stellar subdivisions at all such~$u$, see Figure~\ref{f:over}.
Then add edges $ab\in E_{A'}$ one by one proceeding from $a$ to~$b$,
and making stellar subdivisions at all intersection points as in Step~1.
A the end, we obtain a refinement $B'$ of~$A'$.

\begin{figure}[hbt]
	\includegraphics[width=11.5cm]{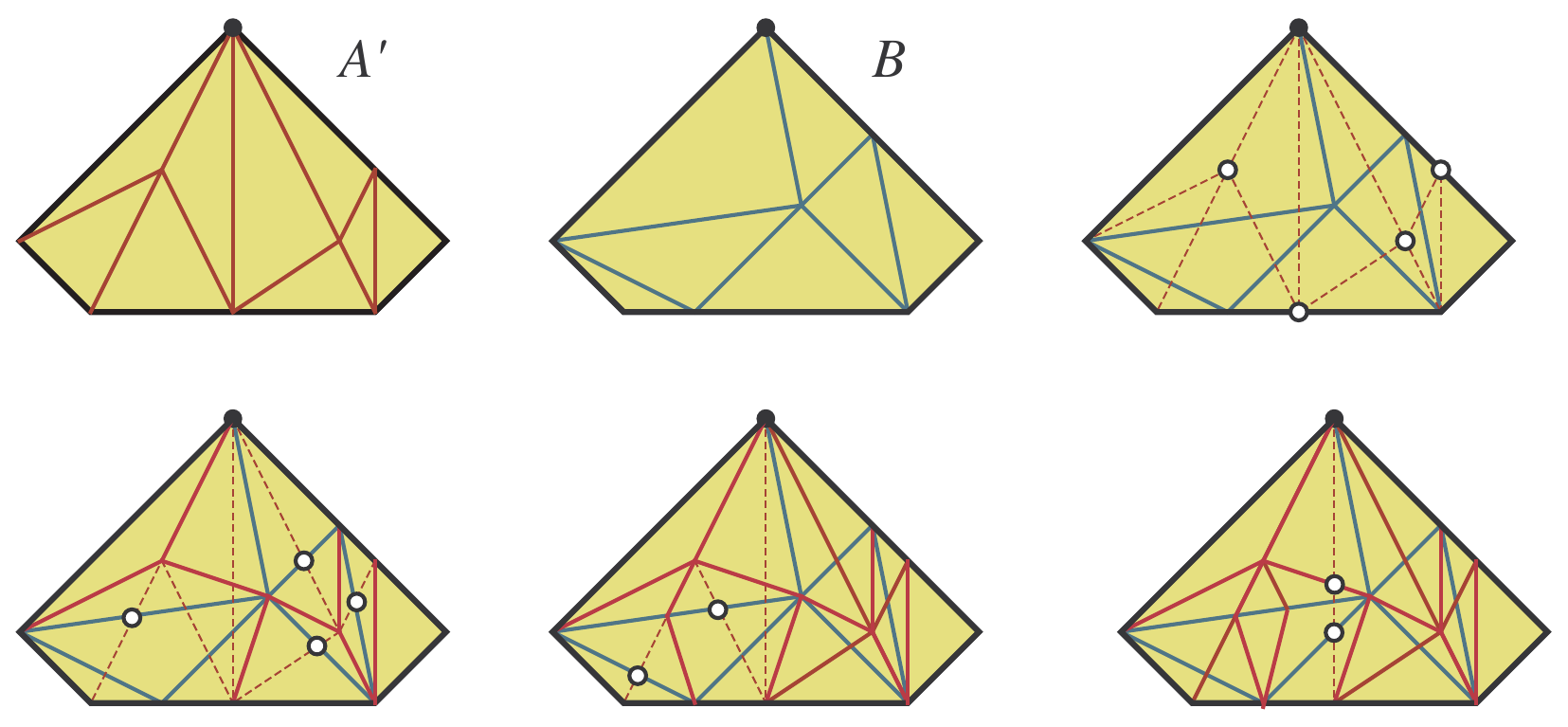}
	%\vskip-.25cm
	\caption{Fin triangulation $A'$, triangulation~$B$, and a sequence
		of stellar subdivisions in Step~2.}
	\label{f:over}
\end{figure}

\medskip

\nin
\defng{\bf{Step~3.}} Compute a \emph{shedding sequence}
$$
Q=Q_0 \, \to \, Q_1 \, \to \, Q_2 \, \to \, \ldots \, \to \, Q_\ell = v
$$
given by the shedding routine as in the proof of Lemma~\ref{l:fin1},
with $T\gets A'$ and $S$ a conic triangulation over vertices in \ts
$\ptl Q \cap V_{A'}$.  Here $D_i := Q_{i-1}\sm Q_i$ is either a triangle
or a union of two triangles obtained by collapsing an edge $u_iw_i$, where
$w_i \in \ptl Q_{i-1}$.

Note that each region $D_i$ is star-shaped at~$u_i$, that $A'$ restricted
to $D_i$ is a stripe fin anchored at~$u_i$, and that $B'$ restricted
to~$D_i$ is a fin anchored at~$u_i$.  For $i=1,2,\ldots, \ell$ in this order,
use Lemma~\ref{l:fin1}, with $S\gets A'|_{D_i}$ and $T\gets B'|_{D_i}$
to obtain a stellar subdivision of $A'$ which coincides with $B'$ on~$D_i$\..

While the stellar subdivision of $D_i$ is constructed by the shedding
routine, some stellar subdivisions will be made for vertices
$z\in D_i \cap Q_i$ on the boundary of both regions.  In these cases,
new edges $zv$ are added to $A'$.  Clearly, the restriction of $A'$
to~$Q_i$ remains a stripe fin anchored at $v$.  Proceed by induction
on~$i$ to obtain~$B'$ as a stellar subdivision of~$A'$.
At the end, we obtain \. $A \lhd A' \lhd B'$ \. and \. $B \lhd B'$,
as desired.

\begin{rem}\label{r:Zeeman}
	Note that Step~2 is a variation on the standard argument which holds in higher
	dimension, see e.g.\ \cite[Lemma~4, p.~8]{Zee63} and \cite{Gla70}:
	
\begin{lemma}
Given $A$ and $B$ simplicial complexes with the same underlying space, 
we can apply stellar subdivisions to $B$ until it refines $A$.
\end{lemma}
\end{rem}

\medskip

\subsection{Proof of Theorem~\ref{t:Oda-plane-quant}}\label{ss:plane-quant}
Note that in Step~1, the number of added intervals is at
most~$m$.  They intersect at most $3m-6$ edges, giving the total of at most
$m(3m-6) < 3m^2$ vertices in~$A'$. This step uses at most $3m^2$ stellar
subdivisions.
In Step~2, the number of vertices in~$B'$
satisfies
$$
|V_{B'}| \, \le  \, |V_{B}| + |V_{A'}|  + |E_{A'}| \cdot |E_{B}| \, \le \,
3m^2+n + \big(3(3m^2-6)-6\big)(3n-6) \, \le \,  27\ts m^2n.
$$
Thus, Step~2 uses at most \. $|V_{A'}|+ |E_{A'}| \cdot |E_{B}|\le  27\ts m^2n$ \.
stellar subdivisions from $B$ to~$B'$.

Finally, the number of stellar subdivisions used in Step~3 is at most
% the number of vertices in~$B'$, i.e.\ at most
\. $|V_{B'}|\le 27 \ts m^2n$.
Summing over Steps~$1$ and~$3$, the number of stellar
subdivisions from $A$ to $B'$ is at most \. $3\ts m^2 +  27 \ts m^2n \le  30\ts m^2 n$.
This completes the proof of Theorem~\ref{t:Oda-plane-quant}.  \qed

\medskip

\section{General algorithm, preparation}

Let us reintroduce some of the notions of the previous section in a more
general form. Since the algorithm involves a rather delicate induction process,
we also separate out those parts that are not part of that inductive process.
Here both $A$ and $B$ that are simplicial complexes realized in $\rr^d$,
triangulating the same polyhedron (or in other words, sharing the same
underlying space).

\subsection{Anchors, fins, stripes and scales}
A polyhedral complex $T$ is \defn{starshaped} if it is starconvex
with respect to~$v$, which is called the \defn{anchor}.
The \defn{horizon} of $T$ is the subset of $T$ defined by points $x$ in~$T$
such that the segment $[x,t]$ lies within $T$, but no line segment strictly
containing $[x,t]$ is contained in~$T$. We call $T$ a \defn{fin} if the
horizon is a subcomplex of $T$ (or equivalently, if the horizon is closed.)
	
We denote this horizon by $\hr_v(T)$. The complex $T$ is called \defn{stripe}
if it coincides with $v\ast \hr_v(T)$, and the latter complex is also called
the \defn{stripe} of~$T$.
A related, and central notion is that of {\em scaled fins}.  A fin $T$ is
\defn{scaled} if the radial projection \. $T\setminus \{v\}\rightarrow \hr_v(T)$ \.
takes every simplex of $T$ that is \emph{not} \. $v$, to a simplex of $\hr_v(T)$.
	
\subsection{Shellings and sheddings}
We say a vertex $w$ in $\hr_v(T)$ is \defn{exposed} if there is an edge $E=E_w$
of $T$ \emph{not} in $\hr_v(T)$ such that $\st_E T = \st_w T$. 
We say that this exposed vertex is \defn{directed}, 
if the edge $E$ is contained in the convex hull of $v$ and $w$.
We say in this case that $T$ has a \defn{shedding} to $T-w$, the maximal subcomplex
of $T$ not containing~$w$. The vertex $w'=E- w$ is also called the
\defn{shedding vertex}.
We say $T$ is \defn{sheddable} if there is a sequence of sheddings such that
the $T$ is reduced to a vertex. A useful example is the following.
	
\begin{ex}
If $T$ is a scaled fin whose stripe is a simplex, then it has a shedding.
\end{ex}

The following observation is useful:

\begin{prop}\label{prp:shedtostellar}
Consider a sheddable scaled fin $T$. Then $T$ is a stellar subdivision of its stripe.
\end{prop}
	
\begin{proof}
Consider the shedding vertices in their natural order, and perform stellar 
subdivisions at these vertices in precisely that order.  
This transforms the stripe into the fin.
\end{proof}	
%We call such fins \defn{unicorns}. In fact, their sheddings are \defn{anchored}:
%the edges $E$ of the shedding can be chosen to contain the anchor $v$ in their affine hull.

We now introduce the notion of \emph{semishedding}. 
We say a face $F$ of $\hr_v(T)$ is \defn{exposed} \ts 
if there is a face $F'$ containing $F$ as a codimension one face s.t.\ \ts 
$\st_F T = \st_{F'} T$, and the shedding is directed if $F'-F$ lies in the 
interior of the convex hull of $F$ and $v$, or coincides with~$v$. 
We then say $T$ has a \defn{semishedding} to $T-F$, the maximal subcomplex
of $T$ not containing~$F$. The vertex $w'=F'- F$ is also called the
\defn{semishedding vertex}, and $T$ is \defn{semisheddable} \ts 
if it can be reduced to a vertex using semishedding steps.
We say $T$ is \defn{sheddable} if there is a sequence of sheddings such that
the $T$ is reduced to a vertex.

Recall finally the notion of \defn{shellability} of a simplicial complex~$S$.
Let $F$ be a facet in~$S$ and let $R$ be the complex consisting of
the remaining facets. Suppose $R$ and $F$ intersect in a subcomplex of
$\partial F$ of uniform dimension equalling that of the latter. In our case, 
as we are dealing with simplicial complexes,
this is the neighborhood of a face of $\partial F$.
A transformation from $S$ to~$R$ is called a \defn{shelling step}.
We say that~$S$ is \defn{shellable} if there is a sequence of
shelling steps which reduces $S$ to a single facet. We have the following fact:

\begin{lemma}\label{lem:pullexist}
Consider a subdivision $S$ of the simplex $\Delta$, and any 
generic point $p$ in $\Delta$. Then there exists an iterated 
stellar subdivision of $\Delta$ that is shellable, 
and such that all intermediate complexes are fins 
with respect to~$p$.
\end{lemma}

\begin{proof} 	
Recall that after sufficiently many stellar subdivisions, 
the triangulation $S$ becomes \emph{regular} \cite{AI15}, i.e.,  
there is a convex piecewise linear function whose domains of 
linearity are exactly the faces of the subdivision~$S'$ of~$S$. 
In other words, we can lift $S'$ to be the boundary of 
a convex polyhedron.

We now use the following Brugesser--Mani trick in~\cite{BM71}.   
Pick a generic point~$p$ on this lifted surface, and move it 
along a half-line~$\ell$ to infinity away from the surface 
(in the apt imagery of \cite[$\S$8.2]{Zie95},  
``launch a rocket upwards'').  Record the order of hyperplanes 
spanned by the facets of~$S'$ encountered along~$\ell$.  This order, 
when seen on $S'$, gives the desired shelling that is star-convex 
with respect to the starting point~$p$.
\end{proof}

\begin{rem}
Lemma~\ref{lem:pullexist} also follows from \cite[Thm~A]{AB17}, 
which states that the triangulation $S$ of $\Delta$ 
becomes shellable after two barycentric subdivisions.
Note that in $\rr^d$, each barycentric subdivision is a composition
of stellar subdivisions: first in all simplices of dimension~$d$, then in all
simplices of dimension~$(d-1)$, etc.  In fact, it follows from the
proof in \cite{AB17}, that the resulting shellable triangulation~$T$
remains strictly star-shaped at~$p$ throughout the shelling.  
Since this result is not explicitly stated, we include
a simple alternative proof above.  However, if one is interested 
in minimizing the number of stellar subdivisions (see~$\S$\ref{ss:finrem-comp}), 
this approach is substantially more efficient.  
\end{rem}

%		
%	\begin{figure}[hbt]
%	\includegraphics[width=15.8cm]{shelling-eps-converted-to}
%	%\vskip-.25cm
%	\caption{Step by construction of a triangulation $S'$ with
%		a pulling towards~$w$, starting from a shelling of a
%		triangle~$\De$ relative to the opposite face~$D$.
%		The triangles are pulled toward the corresponding vertices {\small \bf 1-4}.
%		For example, the pulling locus of vertex~{\small \bf 3} consists of
%		four triangles subdividing triangle~$3$.}
%	\label{f:shelling}
%\end{figure}

\medskip

\section{General case of the weighted strong factorization theorem} \label{s:higher}

We now finalize the proof of the proof, first making some observations and reductions.
	
\subsection{Preparation: triangulations of simplices}\label{ss:higher-simplices}

For the weighted factorization theorem, we are interested in two geometric 
simplicial complexes with the same underlying space.  For simplicity,  
we can assume that the underlying geometric complex is a simplex.  
Indeed, let $X\ssu \rr^d$ be a geometric complex.  We can assume that $X$ is embedded
into a simplex, possibly of larger dimension.  We now use the following
standard result.

\begin{lemma}[{\rm Bing's extension lemma, \cite[$\S$I.2]{Bing83}}] \label{l:Bing}
Let $X\ssu \De$ is a geometric complex embedded in a simplex.
Then there is a triangulation of $\Delta$ that contains $X$ as a subcomplex.
\end{lemma}
	
Let us remark that Bing only states this lemma for $3$-dimensional complexes,
but his proof works in general.  From this point on, we start with two triangulations 
of the simplex, and prove that they do, in fact, have a common stellar 
subdivisions. 
%This is the framework going forward.

\bigskip
	
\subsection{Stripes and scales: Scaling Algorithm}\label{ss:higher-scales}
In this section we present an algorithm that scales a fin. It is one of the
key issues that is more difficult in higher dimensions compared to the
planar case, though the algorithm also works in the planar case.
Formally, we prove the following technical result:
	
\begin{prop}\label{prp:scalefin}
Let $T\ssu \rr^d$ be a triangulation of a $d$-simplex $\De$, and let $v$ be a generic interior point of $\De$.
Then $T$ has an iterated stellar subdivision $T'$ that is also a scaled fin anchored at~$v$.
% \smallskip
% 
% \noindent 
\. Moreover, we can choose $T'$ so that it has a shedding with respect to that anchor.
\end{prop}

In here and what follows, the genericity of $v$ is the one that guaranteed to exist by 
Lemma~\ref{lem:pullexist}.
Let us introduce an important notion in form of a lemma.
	
\begin{lemma}[{\rm Refining scalings and ray-centric subdivisions}{}]
\label{lem:raycenter}
If $T$ is a scaled fin with anchor $v$, and $H=\hr_v(T)$ its horizon,
and if furthermore $H'$ is any stellar subdivision of~$H$,
then some stellar subdivision $T'$ of $T$ has horizon~$H'$.
% 
% \smallskip
% 
% \noindent 
\. Moreover, if \ts $T$ has a shedding, then $T'$ can be chosen to have a shedding as well.
\end{lemma}
	
\begin{proof}
We may assume that $H'$ is obtained from $H$ by a single stellar subdivision,
introducing a vertex $p$ to $H$. Consider the line segment~$pv$.
Consider the faces of $T$ it intersects transversally, that is, 
in a set of dimension~$0$, and order them from $p$ to~$v$. 
Perform stellar subdivisions at these points in this order and 
observe that this process preserves sheddability.
\end{proof}

The proof above defines a subdivision which we call the
\defn{ray-centric subdivision} of~$T$ at the segment~$pv$.

\smallskip
	
Let us now introduce the following notion of \defn{partial scalings}.
A subcomplex $T$ of $\De$ with anchor $v$ is \defn{scaled}
if the radial projection $\varrho$ of $T-v$ to $\De-v$ has
the property that if the relative interiors of $\varrho(\sigma)$
and $\varrho(\tau)$ intersect for any two faces $\sigma$ and $\tau$ of $T$,
then they coincide.
	
Equally useful is the notion of the \defn{upward scaling} of $T$:
If in the above setting, $\varrho(\sigma)$ and $\varrho(\tau)$
have intersecting relative interiors, then $\sigma$ is
in the convex hull of $\{v\}$ and $\tau$ or vice versa.

The notions of sheddings and semisheddings extend as follows.  
We call $T$ a \defn{halfstar} with anchor $v$ if every line through~$v$ 
intersects $T$ in a convex subset, and horizon is the collection of points 
in these sets furthest away from $v$. The \defn{shore} is on the other hand 
those points closest to $v$. We call $T$ a \defn{halffin} if both sets are closed. 
We call a halffin $T$ \defn{sheddable} (resp., \defn{semisheddable}) if shore 
and horizon coincide, or there exists a sheddable (resp., semisheddable) face 
in the horizon and its removal results in a sheddable (resp., semisheddable) 
complex.

\smallskip

We now present an algorithm that proves Proposition~\ref{prp:scalefin}.
	
\medskip

\begin{algSc} \,

\smallskip
\nin
{\em \underline{Input}:}
A triangulation $T$ of a simplex $\Delta$ and generic interior point $v$.

\smallskip

\nin
{\em \underline{Output}:}
A striped stellar subdivision of $T$.
\end{algSc}

\medskip

\nin
\defng{\bf{Step~1.}} As observed in the proof of Lemma~\ref{lem:pullexist}
can assume that $T$ is shellable in such a way that the intermediate
complexes are strictly convex with respect to~$v$. In particular, all simplices are in general position with respect to $p$: a simplex of positive codimension does not contain $v$ in its affine hull.
	
\medskip

\nin	
\defnb{\bf{Subroutine: Upper subdivision of a simplex.}} \.
Consider a simplex $d$-simplex $\sigma$ contained in the $d$-simplex $\De$,
but not intersecting $v$. It is not upward scaled usually, but we can force this easily:
	
The part $L_\sigma$ of $\partial \sigma$ facing $\De$ (the light side illuminated
by the light source~$v$), and the part $D_\sigma$ facing away, project to the same
set in $\De-v$ along $\varrho$. The common subdivision of those two images contains
a unique vertex $s$ that is not a vertex of $L_\sigma$
(if not, $\sigma$ is already upward scaled).

Consider the preimage $s'$ of $s$ in $L_\sigma$. Perform a stellar subdivision
of $\sigma$ at $s$. We call this the \defn{upper stellation} of $\sigma$
at the \defn{upper center} $s$.  \hfill \defnb{{$\mathbf{\blacksquare}$}}

\begin{figure}[h!bt]
\includegraphics[width=12.1cm]{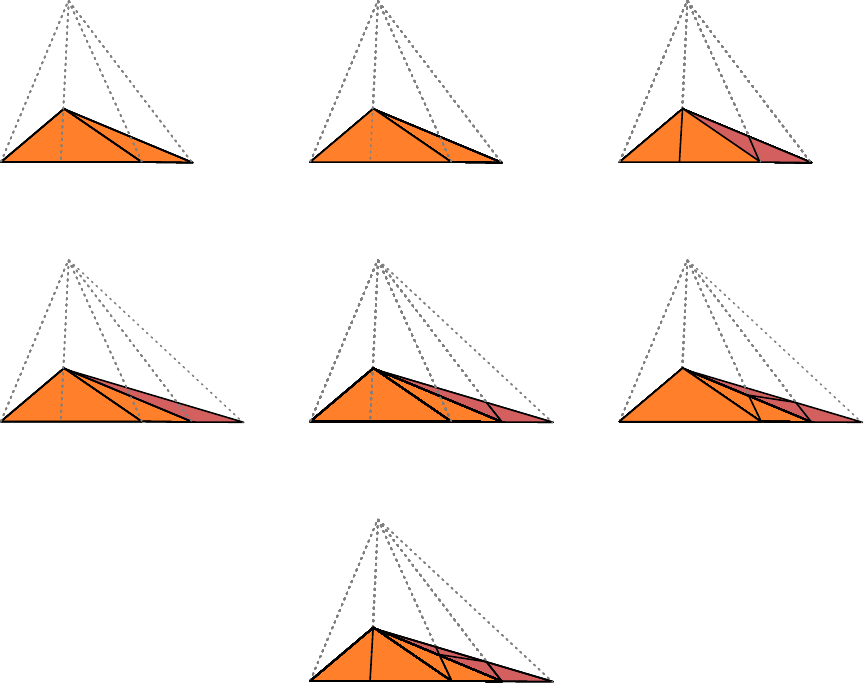}
	%\vskip-.25cm
\caption{Illustration of Step~2 of the scaling algorithm. We start with an initial subcomplex, and record the stellar subdivisions. We then add another simplex, upward stellate, and then perform the stellar subdivisions for the old complex, keeping upper scaling in the new facet as we do so. The result is upward scaled.}
\label{f:shed2}
\end{figure}

If $\sigma$ intersects $v$, then it contains $v$, and the upward stellar subdivision is simply the stellar subdivision at $v$.

\medskip

\nin
\defng{\bf{Step~2.}} We iterate over~$i$.  Let $T_i$ denote the complex of the
first~$i$ facets in the shelling of~$T$. Assume that we already know,
by induction on~$i$, how to find stellar subdivisions to make $T_i$ scaled and sheddable, turning it into a new complex $T_i'$.
Record the stellar subdivision steps in a list~$\mathcal{S}_i$\ts.
We now find a new series of subdivision steps to make $T_{i+1}$ scaled as follows.	

%\smallskip

Let $\sigma$ denote the next facet in the shelling.
First, perform an upper stellation of~$\sigma$.  Next, perform the
subdivision steps in the list $\mathcal{S}_i$, one by one, applied to $T_{i+1}$. We examine the steps one by one, injecting more stellar subdivisions if needed:

If the subdivision is in $\sigma$, then this may introduce simplices in $\sigma$
that are not upward scaled. Pick a simplex~$\tau$ that is no longer upward scaled.
Let $s_\tau$ be the upper center, and perform a ray-centric subdivision with
respect to the segment $sv$.  Otherwise, if the subdivision is not in~$\sigma$, do nothing, that is, proceed to examining the next element of $\mathcal{S}_i$.

%\smallskip
Repeat this for all simplices~$\sigma$ whose upward scaling is now violated.
The new triangulation is upward scaled: restricted to the underlying set of $T_i$, it coincides with $T_i'$.

Observe in addition that it is semisheddable if $T_i'$ was: We can remove using shedding steps until we reach the subdivision of the new facet $\sigma$, which we deformed using an upper stellation, and upper stellations in general position are semisheddable. After this, we performed ray-centric subdivisions, which preserve semisheddability.

\medskip

\nin
\defng{\bf{Step~3.}} We now make the following observation:
	
\begin{lemma}
	Any upward scaled triangulation has a scaled stellar subdivision (with the same stripe). If the complex was semisheddable, the resulting complex can be chosen to be sheddable.
\end{lemma}
	
\begin{proof}
Consider a maximal simplex $\sigma$ of a simplicial complex $T$ in~$\De$
that is \emph{not} scaled, but such that there is no simplex of~$T$
strictly contained in the convex hull of $v$ and $\sigma$ with the same property.
Then there is a vertex in the relative interior of $L_\sigma$ whose
radial projection~$z$ to $D_\sigma$ is not a vertex of the latter.
Perform a stellar subdivision of $T$ at~$z$. The result is still
upward scaled, but the restriction to $\sigma$ is now scaled as well.
Repeat until a scaled stellar subdivision is obtained.
\end{proof}

We use the algorithm in the proof of this lemma to turn the upward scaling into a scaling.
This gives subdivision steps to scale $T_{i+1}$\ts. Repeat Steps~2 and~3 until $T$ is scaled.
This finishes the description of the scaling algorithm and proves Proposition~\ref{prp:scalefin}.  \hfill \defng{{$\mathbf{\blacksquare}$}}

\begin{figure}[hbt]
		\includegraphics[width=12.1cm]{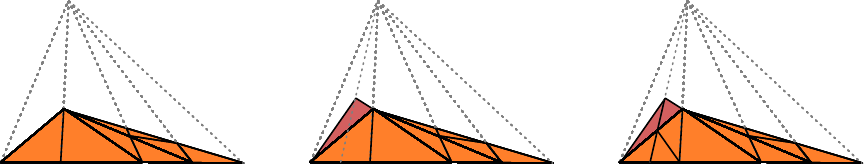}
		%\vskip-.25cm
		\caption{Step~3 of the scaling algorithm turns an upward scaling into a scaling.}
		\label{f:shed3}
\end{figure}	
	
\smallskip

\subsection{Common stellar subdivisions in the simplex: Injection algorithm}
We now can finalize the proof of the weighted strong factorization theorem (Theorem~\ref{t:Oda-space}). We provide the following algorithm

\smallskip

\begin{algFS}
	\ {\em \underline{Input}:}
	$A$ is a shellable simplicial complex of dimension $d$, and $B$ is a refinement of $A$. \\
	{\em \underline{Output}:}
	A common iterated stellar subdivision of $A$ and~$B$.
\end{algFS}

\nin
\defng{\bf{Description of the  FinStar Algorithm.}} \.
By Lemma~\ref{lem:pullexist}, we can assume that $A$ is subdivided to be shellable. Moreover, we may assume that $B$ refines $A$.

Now, order the facets $F_i$ one by one, in their shelling order. Let moreover $v_i$ be generic interior points in each $F_i$.

\medskip

\nin
\defng{\bf{Step~1.}}
Perform stellar subdivisions in $A$ at the points $v_i$.

\medskip

\nin
\defng{\bf{Step~2.}}
Pick the largest $i$ such that when restricted to the complex $T_i$ of the first $i$ facets in the shelling order, the triangulations $A'$  and $B'$ coincide.

Consider the facet $F_{i+1}$. Restricted to this facet, $A'$ is a stripe. Use the scaling algorithm to make $B'|_{F_{i+1}}$ a sheddable, scaled fin with anchor $v_i$.

Consider now the horizons $\hr_{v_{i+1}} A'|_{F_{i+1}}$ and $\hr_{v_{i+1}} B'|_{F_{i+1}}.$ By induction on the dimension, they have a common stellar subdivision. Hence, we can apply Lemma~\ref{lem:raycenter} and apply stellar subdivisions until $A'|_{F_{i+1}}$ is the stripe of the scaled fin $B'|_{F_{i+1}}$.

\medskip

\nin
\defng{\bf{Step~3.}}
Use Proposition~\ref{prp:shedtostellar}, applied to $B'|_{F_{i+1}}$, to find a stellar subdivision of $A'|_{F_{i+1}}$ that coincides with $B'$. Now, the triangulations $A'$ and $B'$ coincide on $T_{i+1}$.  Return to Step~2 and repeat until a common subdivision it obtained. \hfill \defng{{$\mathbf{\blacksquare}$}}

\smallskip

\begin{proof}[Proof of Theorem~\ref{t:Oda-space}]
Now, recall we may assume that $A$ and $B$ refine a simplex (by Lemma~\ref{l:Bing}), and that $A$ is shellable by Lemma~\ref{lem:pullexist}, and that $B$ refines $A$. Apply the FinStar algorithm.
	
Combining the algorithms, routines and subroutines,
we obtain the first part of the theorem.
For the second part, note that if all vertices have coordinated over~$K$,
then so do all hyperplanes and their intersections.  This implies that
the whole construction is defined over~$K$, as desired.
\end{proof}

\medskip

\section{Proof of Alexander's conjecture}\label{s:Alex}
	
It was shown in \cite{AM03}, that Theorem~\ref{t:Alex} follows from Theorem~\ref{t:Oda-space}.
We include a short proof for completeness.

\smallskip

\nin
{\em Proof.}
Let $A$ and $B$ be two simplicial complexes, and let $\varphi: A \rightarrow B$ be a PL homeomorphism.
Observe that by pulling back the triangulation of $B$ to~$A$, we can find a subdivision $A'$ of~$A$
such that $\varphi: A' \rightarrow B$ is linear on every face of~$A'$.
	
Observe now that if $A''$ is a stellar subdivision of $A$ that refines $A'$,
then $\varphi: A'' \rightarrow B$ is linear as well.
Hence, we can think of $B$ as a geometric simplicial complex, and $A''$ as a
geometric subcomplex.  Apply the strong factorization theorem (Theorem~\ref{t:Oda-space}),
to obtain a common stellar subdivision of $A''$ and $B$, and therefore of $A$ and~$B$.
\qed
	
\medskip

\section{Final remarks and open problems} \label{s:finrem}

\subsection{Unweighted Oda's conjecture} \label{ss:finrem-unweighted}
Now that the weighted Oda conjecture is settled (Corollary~\ref{c:birational}),
is natural to ask about the \defng{unweighted Oda conjecture} \cite{Oda78}.
This conjecture concerns lattice fans in an ambient lattice $\Lambda$.

Consider two simplicial, unimodular fans with the same support.
Here by \defn{unimodular} \ts we mean that the fan is generated by lattice vectors,
and that the lattice points in each defining ray~$\rho$ of a simplicial cone~$\sigma$,  
span the sublattice generated by $(\mathrm{span}~\sigma) \cap \Lambda$. 
Consider now only \defn{smooth} stellar subdivisions of simplices: 
where as in stellar subdivisions we introduced a new vertex~$z$ at 
arbitrary coordinates, here we only allow to introduce the lattice 
point \. $z=\sum e_\rho$\., where the summation is over $\rho$
defining ray of~$\sigma$, and $e_\rho$ is the lattice point 
in~$\rho$ generating $\rho \cap \Lambda$.

\begin{question}\label{q:oda-smooth}
Consider two unimodal fans of the same support. Are there two common iterated
stellar subdivisions at \emph{smooth} centers?
\end{question}

The algorithm as present does not give this result. It is easy to see that the
Scaling algorithm can be modified to work with respect to the restrictions to
smooth subdivisions. Unfortunately, we do not know how to modify the FinStar
algorithm.

\subsection{Toroidalization and general varieties}

It is natural to ask whether the Oda's program for toric varieties extends
to general varieties connected by birational maps. This is an open problem,
and subject of the \defng{toroidalization conjecture} \cite{AMR99}.
Without getting technical, the question is whether a birational morphism
of varieties can be turned, after blowups at smooth centers,
into a morphism of toric varieties.  Thanks to the work of Cutkosky
\cite{Cut07}, this is illuminated for varieties up to dimension~$3$.

\subsection{Distance between topological triangulations}  \label{ss:finrem-comp-top}
For PL manifolds in dimensions $d\ge 4$, the problem of homeomorphism
is undecidable \cite{Mar58}.  This implies that the number of stellar
subdivisions needed in Theorem~\ref{t:Alex} is not computable.
We refer to \cite{AFW15,Lac22} for detailed surveys of decidability
and complexity of the homeomorphism and related problems.

%	Note that in dimension $d=3$, the bound in \cite[Thm~4]{AI15} is
%	triply exponential in the number of faces of the simplicial complex.

\subsection{Distance between geometric triangulations}  \label{ss:finrem-comp}
There are few results on distances between geometric triangulations
under different types of flips.  We refer to \cite{San06} for a survey
on bistellar flips when the graph is disconnected in dimension $d\ge 5$
(when new vertices cannot be added).  When both stellar flips and
reverse stellar flips are added, a recent upper bound in \cite{KP21}
is exponential in~$d$ and polynomial in the number of simplices
(for fixed~$d$).  Our preliminary calculations show that for
triangulations of simplices, the bound we give is roughly of
the same order.
		
\subsection{Da Silva and Karu's algorithm}  \label{ss:finrem-DK}
Note that our choices of stellar subdivisions are asymmetric
with respect to triangulations and uses a delicate ordering
given by the shedding routine.  In~\cite{DK11}, the authors
proposed an algorithm for common stellar subdivision and
conjectured that it works in finite time.  It would be
interesting to see if our proof of Theorem~\ref{t:Oda-space}
helps to resolve the conjecture.  Note that this would
simultaneously give a positive answer to Question~\ref{q:oda-smooth},
since the algorithm of Da Silva and Karu uses only
smooth stellar subdivisions.

\subsection{Dissections}  \label{ss:finrem-dis}
For dissections of polyhedra, there is a natural notion of
\emph{elementary dissection} which consists of dividing a simplex
into two. Motivated by applications to scissors congruence, Sah
claimed in \cite[Lemma~2.2]{Sah79} without a proof,
that every two dissections of a geometric complex have a common
dissection obtained as composition of elementary dissections.
It would be interesting to see if the approach in this paper
can be extended to prove this result.
	
Note that both stellar subdivisions and bistellar flips are
compositions of elementary dissections and their inverses;
these are called \emph{elementary moves}.
Ludwig and Reitzner proved in \cite{LR06} that
all dissections of a geometric complex are connected by
elementary moves.  For convex polygons in the plane, see a
self-contained presentation of the proof in \cite[$\S$17.5]{Pak10}.
We refer to \cite{LR06} also for an overview of the previous
literature, and for applications to valuations.

	\smallskip

\subsection*{Acknowledgements}
We are grateful to Dan Abramovich, Sergio Da~Silva, Joaquin Moraga,
Nikolai Mn\"ev and Tadao Oda for helpful remarks.
The first author is supported by the ISF, the
Centre National de Recherche Scientifique, and the ERC.
The second author is partially supported by the~NSF.

\vskip.8cm

	{\footnotesize

\vskip.5cm
	}

\end{document}